\documentclass[12pt]{amsart}
\usepackage[margin=1in]{geometry}
\usepackage{amsmath,amsfonts,amsthm,amssymb,bbm}
\usepackage{graphicx,color,dsfont}
\usepackage{enumitem}
\usepackage{fourier}

\begin{document}

\newtheorem{theorem}{Theorem}
\newtheorem{lemma}{Lemma}
\newtheorem{proposition}{Proposition}
\newtheorem{rmk}{Remark}
\newtheorem{example}{Example}
\newtheorem{exercise}{Exercise}
\newtheorem{definition}{Definition}
\newtheorem{corollary}{Corollary}
\newtheorem{notation}{Notation}
\newtheorem{claim}{Claim}

\newtheorem{dif}{Definition}

 \newtheorem{thm}{Theorem}[section]
 \newtheorem{cor}[thm]{Corollary}
 \newtheorem{lem}[thm]{Lemma}
 \newtheorem{prop}[thm]{Proposition}
 \theoremstyle{definition}
 \newtheorem{defn}[thm]{Definition}
 \theoremstyle{remark}
 \newtheorem{rem}[thm]{Remark}
 \newtheorem*{ex}{Example}
 \numberwithin{equation}{section}

\newcommand{\vertiii}[1]{{\left\vert\kern-0.25ex\left\vert\kern-0.25ex\left\vert #1
    \right\vert\kern-0.25ex\right\vert\kern-0.25ex\right\vert}}

\newcommand{\R}{{\mathbb R}}
\newcommand{\C}{{\mathbb C}}
\newcommand{\U}{{\mathcal U}}
\newcommand{\norm}[1]{\left\|#1\right\|}
\renewcommand{\(}{\left(}
\renewcommand{\)}{\right)}
\renewcommand{\[}{\left[}
\renewcommand{\]}{\right]}
\newcommand{\f}[2]{\frac{#1}{#2}}
\newcommand{\im}{i}
\newcommand{\cl}{{\mathcal L}}
\newcommand{\ck}{{\mathcal K}}

\newcommand{\al}{\alpha}
\newcommand{\vro}{\varrho}
\newcommand{\be}{\beta}
\newcommand{\wh}[1]{\widehat{#1}}
\newcommand{\ga}{\gamma}
\newcommand{\Ga}{\Gamma}
\newcommand{\de}{\delta}
\newcommand{\ben}{\beta_n}
\newcommand{\De}{\Delta}
\newcommand{\ve}{\varepsilon}
\newcommand{\ze}{\zeta}
\newcommand{\Th}{\Theta}
\newcommand{\ka}{\kappa}
\newcommand{\la}{\lambda}
\newcommand{\laj}{\lambda_j}
\newcommand{\lak}{\lambda_k}
\newcommand{\La}{\Lambda}
\newcommand{\si}{\sigma}
\newcommand{\Si}{\Sigma}
\newcommand{\vp}{\varphi}
\newcommand{\om}{\omega}
\newcommand{\Om}{\Omega}
\newcommand{\ra}{\rightarrow}

\newcommand{\ro}{{\mathbf R}}
\newcommand{\rn}{{\mathbf R}^n}
\newcommand{\rd}{{\mathbf R}^d}
\newcommand{\rmm}{{\mathbf R}^m}
\newcommand{\rone}{\mathbb R}
\newcommand{\rtwo}{\mathbf R^2}
\newcommand{\rthree}{\mathbf R^3}
\newcommand{\rfour}{\mathbf R^4}
\newcommand{\ronen}{{\mathbf R}^{n+1}}
\newcommand{\ku}{\mathbf u}
\newcommand{\kw}{\mathbf w}
\newcommand{\kf}{\mathbf f}
\newcommand{\kz}{\mathbf z}

\newcommand{\N}{\mathbf N}

\newcommand{\tn}{\mathbf T^n}
\newcommand{\tone}{\mathbf T^1}
\newcommand{\ttwo}{\mathbf T^2}
\newcommand{\tthree}{\mathbf T^3}
\newcommand{\tfour}{\mathbf T^4}

\newcommand{\zn}{\mathbf Z^n}
\newcommand{\zp}{\mathbf Z^+}
\newcommand{\zone}{\mathbf Z^1}
\newcommand{\zz}{\mathbf Z}
\newcommand{\ztwo}{\mathbf Z^2}
\newcommand{\zthree}{\mathbf Z^3}
\newcommand{\zfour}{\mathbf Z^4}

\newcommand{\hn}{\mathbf H^n}
\newcommand{\hone}{\mathbf H^1}
\newcommand{\htwo}{\mathbf H^2}
\newcommand{\hthree}{\mathbf H^3}
\newcommand{\hfour}{\mathbf H^4}

\newcommand{\cone}{\mathbf C^1}
\newcommand{\ctwo}{\mathbf C^2}
\newcommand{\cthree}{\mathbf C^3}
\newcommand{\cfour}{\mathbf C^4}
\newcommand{\dpr}[2]{\langle #1,#2 \rangle}

\newcommand{\sn}{\mathbf S^{n-1}}
\newcommand{\sone}{\mathbf S^1}
\newcommand{\stwo}{\mathbf S^2}
\newcommand{\sthree}{\mathbf S^3}
\newcommand{\sfour}{\mathbf S^4}

\newcommand{\lp}{L^{p}}
\newcommand{\lppr}{L^{p'}}
\newcommand{\lqq}{L^{q}}
\newcommand{\lr}{L^{r}}
\newcommand{\echi}{(1-\chi(x/M))}
\newcommand{\chip}{\chi'(x/M)}

\newcommand{\wlp}{L^{p,\infty}}
\newcommand{\wlq}{L^{q,\infty}}
\newcommand{\wlr}{L^{r,\infty}}
\newcommand{\wlo}{L^{1,\infty}}

\newcommand{\lprn}{L^{p}(\rn)}
\newcommand{\lptn}{L^{p}(\tn)}
\newcommand{\lpzn}{L^{p}(\zn)}
\newcommand{\lpcn}{L^{p}(\cn)}
\newcommand{\lphn}{L^{p}(\cn)}

\newcommand{\lprone}{L^{p}(\rone)}
\newcommand{\lptone}{L^{p}(\tone)}
\newcommand{\lpzone}{L^{p}(\zone)}
\newcommand{\lpcone}{L^{p}(\cone)}
\newcommand{\lphone}{L^{p}(\hone)}

\newcommand{\lqrn}{L^{q}(\rn)}
\newcommand{\lqtn}{L^{q}(\tn)}
\newcommand{\lqzn}{L^{q}(\zn)}
\newcommand{\lqcn}{L^{q}(\cn)}
\newcommand{\lqhn}{L^{q}(\hn)}

\newcommand{\lo}{L^{1}}
\newcommand{\lt}{L^{2}}
\newcommand{\li}{L^{\infty}}
\newcommand{\beqn}{\begin{eqnarray*}}
\newcommand{\eeqn}{\end{eqnarray*}}
\newcommand{\pplus}{P_{Ker[\cl_+]^\perp}}

\newcommand{\co}{C^{1}}
\newcommand{\ci}{C^{\infty}}
\newcommand{\coi}{C_0^{\infty}}

\newcommand{\ca}{\mathcal A}
\newcommand{\cs}{\mathcal S}
\newcommand{\cm}{\mathcal M}
\newcommand{\cf}{\mathcal F}
\newcommand{\cb}{\mathcal B}
\newcommand{\ce}{\mathcal E}
\newcommand{\cd}{\mathcal D}
\newcommand{\cn}{\mathcal N}
\newcommand{\cz}{\mathcal Z}
\newcommand{\crr}{\mathbf R}
\newcommand{\cc}{\mathcal C}
\newcommand{\ch}{\mathcal H}
\newcommand{\cq}{\mathcal Q}
\newcommand{\cp}{\mathcal P}
\newcommand{\cx}{\mathcal X}
\newcommand{\eps}{\epsilon}

\newcommand{\pv}{\textup{p.v.}\,}
\newcommand{\loc}{\textup{loc}}
\newcommand{\intl}{\int\limits}
\newcommand{\iintl}{\iint\limits}
\newcommand{\dint}{\displaystyle\int}
\newcommand{\diint}{\displaystyle\iint}
\newcommand{\dintl}{\displaystyle\intl}
\newcommand{\diintl}{\displaystyle\iintl}
\newcommand{\liml}{\lim\limits}
\newcommand{\suml}{\sum\limits}
\newcommand{\ltwo}{L^{2}}
\newcommand{\supl}{\sup\limits}
\newcommand{\df}{\displaystyle\frac}
\newcommand{\p}{\partial}
\newcommand{\Ar}{\textup{Arg}}
\newcommand{\abssigk}{\widehat{|\si_k|}}
\newcommand{\ed}{(1-\p_x^2)^{-1}}
\newcommand{\tT}{\tilde{T}}
\newcommand{\tV}{\tilde{V}}
\newcommand{\wt}{\widetilde}
\newcommand{\Qvi}{Q_{\nu,i}}
\newcommand{\sjv}{a_{j,\nu}}
\newcommand{\sj}{a_j}
\newcommand{\pvs}{P_\nu^s}
\newcommand{\pva}{P_1^s}
\newcommand{\cjk}{c_{j,k}^{m,s}}
\newcommand{\Bjsnu}{B_{j-s,\nu}}
\newcommand{\Bjs}{B_{j-s}}
\newcommand{\Ly}{\cl_+i^y}
\newcommand{\dd}[1]{\f{\partial}{\partial #1}}
\newcommand{\czz}{Calder\'on-Zygmund}
\newcommand{\chh}{\mathcal H}

\newcommand{\lbl}{\label}
\newcommand{\beq}{\begin{equation}}
\newcommand{\eeq}{\end{equation}}
\newcommand{\beqna}{\begin{eqnarray*}}
\newcommand{\eeqna}{\end{eqnarray*}}
\newcommand{\bp}{\begin{proof}}
\newcommand{\ep}{\end{proof}}
\newcommand{\bprop}{\begin{proposition}}
\newcommand{\eprop}{\end{proposition}}
\newcommand{\bt}{\begin{theorem}}
\newcommand{\et}{\end{theorem}}
\newcommand{\bex}{\begin{Example}}
\newcommand{\eex}{\end{Example}}
\newcommand{\bc}{\begin{corollary}}
\newcommand{\ec}{\end{corollary}}
\newcommand{\bcl}{\begin{claim}}
\newcommand{\ecl}{\end{claim}}
\newcommand{\bl}{\begin{lemma}}
\newcommand{\el}{\end{lemma}}
\newcommand{\dea}{(-\De)^\be}
\newcommand{\naa}{|\nabla|^\be}
\newcommand{\cj}{{\mathcal J}}
\newcommand{\ubb}{{\mathbf u}}

\title[ ]{On the  stability of  the dnoidal waves  for the Schr\"odinger - KdV system}

\author{Sevdzhan Hakkaev}
\address{
	 Trakya University, Department of Mathematics, 22030, Edirne, Turkey; Institute of Mathematics and Informatics, Bulgarian Academy of Sciences, Acad. G. Bonchev Str. bl. 8, 1113, Sofia, Bulgaria;
}\email{s.hakkaev@math.bas.bg}

\author[Atanas G. Stefanov]{\sc Atanas G. Stefanov}
\address{ Department of Mathematics,
	University of Alabama - Birmingham,
	University Hall, Room 4005,
	1402 10th Avenue South
	Birmingham AL 35294-1241
	 }
\email{stefanov@uab.edu}



\date{\today}

\begin{abstract}
We study the periodic Schr\"odinger-Korteweg de Vries system.  
We describe the two-parametetric family of $2T$ periodic traveling waves of dnoidal type. The main objective of the paper is to establish their
spectral stability with respect to co-periodic perturbations. In the limit $T\to \infty$, we recover the results of Albert-Angulo, \cite{AA} for the stability of the soliton solutions of this system  on the real line. 
\end{abstract}

\thanks{ Stefanov is partially supported by   NSF-DMS \# 2204788.}

\maketitle

\section{Introduction}
An interaction of long and short nonlinear waves appears in various physical contexts. The short wave is described by nonlinear Schr\"odinger equation and long wave usually described by wave equation with dispersive term. The typical model  is the following  Schr\"odinger-Korteweg de Vries system \cite{FO, KSK, SY}
\begin{equation}\label{1}
  \left\{ 
  \begin{array}{ll}
  iu_{t}+ u_{xx}=c_1 uv+c_2|u|^2u\\
  \\
		v_t+ \frac{1}{2}(v^2)_x+ v_{xxx}=\gamma(|u|^2)_x.
		\end{array} \right.
		\end{equation}
		It is standard to check, that smooth solutions of the the system (\ref{1}) enjoy the following  conservation laws
$$\left\{ \begin{array}{ll}
   E(u,v)=\int (\gamma |u_x|^2+\frac{c_1}{2}v_x^2+\frac{c_2\gamma}{2}|u|^4+c_1\gamma v|u|^2-\frac{c_1}{6}v^3)dx\\
   \\
   G(u,v)=\int c_1v^2dx-2\gamma Im \int u\bar{u_x}dx\\
   \\
   F(u)=\int |u|^2dx.
   \end{array} \right.
   $$
   The Cauchy problem on the whole line case for the system  (\ref{1}) was considered in numerous papers \cite{BOP, CL, GM, Ts}. The Cauchy problem in periodic context for (\ref{1}), including in low regularity Sobolev spaces, was considered in \cite{ACM}. 
   
   For the parameters $c_1=-1,c_2=0, \gamma=-1$,  Chen, \cite{Ch}  has constructed two-parameter family of explicit solitary wave solutions on the line for system \eqref{1}  and considered their orbital stability.  His  approach relies on the  abstract stability theory of solitary waves for Hamiltonian systems with symmetries developed in \cite{GSS2} and on the corresponding conservation laws. The existence and stability results for ground-state solutions to
the system (\ref{1}),  with $c_2=0$,  was obtained in \cite{AA}. Notably, their analysis relies upon the variational structure of the problem, which is subject to two constraints (one for each variable $u,v$), which lead to the existence of a two-parameter family of ground states. 

In the present paper,  we study the following nonlinear  wave system
\begin{equation}
	\label{10}
	 \begin{cases}
		iu_{t}+ u_{xx}+uv=0\\
		v_t+v_x+\beta (v^2)_x+\alpha v_{xxx}=-\frac{1}{2}(|u|^2)_x.
	\end{cases}
\end{equation}
Here $u$ is a complex-valued function, $v$ real-valued function,  and $\al>0, \be>0$.  Clearly, our system is equivalent to \eqref{1} with $c_2=0$,  up to a rescaling.  In fact, our approach is likely to work even in the cases $c_2\neq 0$, but for the sake of simplicity, we do not pursue these issues herein. 

Our aim is to construct periodic waves and then to consider their spectral stability. In order to explain our spectral stability results in detail, we need to linearize the system (\ref{10})
about the periodic traveling wave solutions. Then,  we  obtain the required spectral information
about the operator of linearization and investigate the corresponding Hamiltonian eigenvalue problem. 

Next, we construct special families of  periodic waves.
\subsection{Construction of the periodic waves}
We take the ansatz $u(t,x)=e^{i\om t}e^{i\frac{c}{2}(x-ct)}\varphi(x-ct)$, $ v(t,x)=\psi(x-ct)$, where we assume that $\vp, \psi$ are spatially periodic functions.  Plugging in the system, we get
\begin{equation}
	\label{25}
	\begin{cases}
		\varphi''-\sigma \varphi+\varphi \psi=0\\
		-c\psi'+\psi'+2\beta\psi \psi'+\alpha\psi'''=-\frac{1}{2}(\varphi^2)',
	\end{cases}
\end{equation}
where $\sigma=\om - \frac{c^2}{4}$. Integrating in the second equation yields 
\begin{equation}
	\label{3.2}
	 \begin{cases}
		\varphi''-\sigma \varphi+\varphi \psi=0\\
		-c\psi+\psi+\beta\psi^2+\alpha\psi''+\frac{1}{2}\varphi^2=b.
	\end{cases}
\end{equation}
for some constant $b$.  We look for solutions in the form $\psi=A\varphi^2, A>0$.   In this form, the first equation in (\ref{3.2}) converts as
\begin{equation}\label{3.3}
	-\varphi''+\sigma \varphi-A\varphi^3=0.
\end{equation}
We have the following proposition.
\begin{proposition}
	\label{prop:10}
	Let $c>1,  \om\in \rone: \om>\f{c^2}{4}$, and so $\si:=\om-\f{c^2}{4}>0$.  If $\al: 0<\al<\f{c-1}{4\si}$ and $\be=3\al$,  then the system \eqref{25} has a one parameter family of solutions in the form \\ $(e^{i\om t}e^{i\frac{c}{2}(x-ct)}\varphi(x-ct), \psi(x-ct))$, with
	\begin{equation}
		\label{8} 
			\vp(x)=\vp_0 dn(\ga x, \ka), \psi(x) = A\vp^2(x), \ka\in (0,1)
	\end{equation}
	where
	\begin{eqnarray*}
	A=\frac{1}{2(c-1-4\alpha \sigma)}>0,\ \  \ga^2=\f{\si}{2-\ka^2}, \ \ \vp_0^2=\f{2\ga^2}{A},
	\end{eqnarray*}
also satisfies \eqref{3.2}, with
$$
b=-4\al  (1-\ka^2)  \ga^4<0.
$$
The solution is periodic with period $2\f{K(\ka)}{\ga}$, and we will consider it as defined on an interval $[-T, T]: T=\f{K(\ka)}{\ga}$.
\end{proposition}
The standard proof of Proposition \ref{prop:10} is presented in the Appendix. 

\subsection{Linearized equations and spectral stability}
In order to analyze the stability of the dnoidal waves constructed in Proposition \ref{prop:10}, we need to consider the corresponding linearized problem. 
To this end, let 
$$
u(t,x)=
e^{i \om t}e^{i\frac{c}{2}(x-ct)}(\varphi(x-ct)+p(t,x-ct));  \ \  v(t,x)=\psi(x-ct)+q(t,x-ct).
$$
 Plugging in the system and ignoring all quadratic and higher terms in $p,q$, we get
\begin{equation}\label{4.1}
	\left\{ \begin{array}{ll}
		ip_t-\sigma p+p_{xx}+\varphi q+\psi p=\\
		\\
		q_t-cq_x+q_x+2\beta (\psi q)_x+\alpha q_{xxx}+(\varphi \Re p)_x=0.
	\end{array} \right.
\end{equation}
Separating the real and imaginary part $p=p_1+ip_2$, we get
\begin{equation}\label{4.1a}
	\left\{ \begin{array}{ll}
		p_{1t}=-p_{2xx}-\sigma p_2-\psi p_2\\
		\\
		-p_{2t}=-p_{1xx}+\sigma p_1-\psi p_1-\varphi q\\
		\\
		q_t=\partial_x(-\alpha q_{xx}+cq-q-2\beta \psi q-\varphi p_1).
	\end{array} \right.
\end{equation}
For $\vec{U}=(q,p_1,p_2)$, the system (\ref{4.1a}) can be written in the form
\begin{equation}\label{4.2}
	\vec{U}_t=\cj\cl\vec{U},
\end{equation}
where
\begin{eqnarray}
	\label{4.3}
	\cj &=& \begin{pmatrix} \partial_x & 0 &0\\ 0& 0 &1\\ 0&-1&0 \end{pmatrix}, \; \; \; \cl=\begin{pmatrix} \cl_1 & -\varphi &0 \\ -\varphi & \cl_2&0  \\ 0&0&\cl_2  \end{pmatrix} \\
	\label{4.4}
		\cl_1 &=& -\alpha\partial_x^2+(c-1)-2\beta \psi\\
		\cl_2 &=& -\partial_x^2+\sigma -\psi.
\end{eqnarray}
On a technical point here, the operators $\cl_j, j=1,2$ are self-adjoint on the space $L^2_{per.}[-T,T]$, when considered with the domains $H^2[-T,T]$, with periodic boundary conditions. Accordingly,  we consider the associated operator $\cl$, with $D(\cl)=H^2[-T,T]^3$. 

Before we pass on to the main results in this paper, we need another standard notion, namely spectral stability. 
\begin{definition}
	\label{defi:10} 
	We say that the wave $\vp$ is spectrally stable, if the associated eigenvalue problem 
	\begin{equation}
		\label{e:10} 
		\cj \cl \vec{f}=\la \vec{f},
	\end{equation}
does not have a non-trivial solution $(\la, \vec{f}): \Re\la>0, \vec{U}\neq 0, \vec{f}\in D(\cl)$. Otherwise, we say that the wave $\vp$ is spectrally unstable. 
\end{definition}
\subsection{Main results}

\begin{theorem}
	\label{theo:10} 
	Let $c>1, \om\in \rone,  \om>\f{c^2}{4},  0<\al<\f{c-1}{4\om-c^2}$ be the parameters, as described in Proposition \ref{prop:10}. Let $\vp, \psi$, given by \eqref{8},  be the corresponding dnoidal wave. Then the traveling wave 
	$$
	(e^{i\om t}e^{i\frac{c}{2}(x-ct)}\varphi(x-ct);   \psi (x-ct))
	$$
	  is spectrally stable, in the sense of Definition \ref{defi:10}. 
\end{theorem}
{\bf Remark:} As $T\to \infty$, the waves obtained in Proposition \ref{prop:10} converge top the corresponding $sech$ solitons. The results of Theorem \ref{theo:10} then imply the spectral stability of these waves in the whole line context, recovering the results of Chen, \cite{Ch}. We do not pursue this connection anymore, instead we focus on the periodic problem henceforth. 

Our work is organized as follows. In Section \ref{sec:2}, we start with an outline of some routine preliminary material, followed by a discussion of  the spectral properties of a relevant scalar Schr\"odinger operators. In Section \ref{sec:3}, we analyze the spectral problem \eqref{e:10}, based on the knowledge acquired in the previous section. Interestingly, neither the Grillakis-Shatah-Strauss theory, nor the instability index counting theories prove useful in this context, as the required quantities look impossible to compute explicitly. Instead, we rely  on  establishing some orthogonality identities directly for the entries of  the eigenvalue problem. 
\section{Spectral properties of some relevant scalar Schr\"odinger operators}
 \label{sec:2}
We start with a well-known result in the theory, namely the so-called Weinstein's lemma, see also an extension in \cite{HSS16} (Lemma 1 and Theorem 1). Before we provide the exact statement, we introduce for future reference, the projection operator $P_{\{\xi\}^\perp}$, which projects onto the orthogonal complement of a fixed vector $\xi$. Specifically, 
$$
P_{\{\xi\}^\perp}f=f-\f{\dpr{f}{\xi}}{\|\xi\|^2} \xi,
$$
 
\begin{lemma}[Weinstein, \cite{Wei}]
	\label{le:Wei} 
	
	Let $(A, D(A))$ be a self-adjoint operator acting on a Hilbert space, with exactly one negative simple eigenvalue\footnote{Alternatively, we might say that $A$ has a dimension one negative subspace, spanned by $\eta_0$}, say $-\si^2: A\eta_0=-\si^2 \eta_0$.  Assume also the zero gap condition 
	$$
	A|_{span\{\eta_0, Ker(A)\}^\perp}\geq \de_0>0.
	$$ 
	Assume that there is $\xi_0\in Ker(A)^\perp$, so that $\dpr{A^{-1} \xi_0}{\xi_0}<0$. Then, 
	\begin{equation}
		\label{k:10} 
		A|_{\{\xi_0\}^\perp}\geq 0.
	\end{equation}
Concretely,  $P_{\{\xi_0\}^\perp}A P_{\{\xi_0\}^\perp}\geq 0$, 
\end{lemma}

We now build some necessary preliminary properties of the operators involved in the spectral stability analysis. 
Our first result is about the operator $\cl_2=-\p_x^2+\si - A \vp^2$, which due to the relation \eqref{3.3}, it is in fact the  linearized operator $\cl_-$ for the $dn$ solution of the cubic focusing NLS. More precisely, we state the following well-known result.
\begin{proposition}
	\label{prop:15}
	The operator $\cl_2, D(\cl_2)=H^2_{per.}[-T,T]$ is a non-negative operator. Also,
	$$
	Ker(\cl_2)=span[\vp], \ \ \cl_2|_{\{\vp\}^\perp}>0.
	$$
	Regarding the operator $L:=\cl_2 - 2A  \vp^2=-\p_x^2+\si - 3A \vp^2$, we have the following spectral properties
	\begin{enumerate}
		\item $L$ has exactly one negative eigenvalue, with a positive  eigenfunction,  say $h_0$.
		\item $L[\vp']=0$. Moreover $Ker(L)=span[\vp']$ and
		$L|_{\{h_0, \vp'\}^\perp}>0$.
		\item $L|_{\{\vp\}^\perp}\geq 0$ and in fact, $L|_{\{\vp, \vp'\}^\perp}>0$. That is, there exists $\de>0$, so that
		$$
		\dpr{L h}{h}\geq \de\|h\|^2, \forall h: h\perp \vp, h\perp \vp'
		$$
	\end{enumerate}
\end{proposition}
\begin{proof}
Regarding $\cl_2$, 	$\cl_2[\vp]=0$ is by inspection, $0\in \si(\cl_2)$.
	 As $\vp>0$, this is the bottom of the spectrum, which is a simple eigenvalue, hence $Ker(\cl_2)=span[\vp]$ and $\cl_2|_{\{\vp\}^\perp}>0$.
	
	 Regarding $L$, this is in fact the linearized  operator $\cl_+$ corresponding to the $dn$ solution of the cubic NLS, \eqref{3.3}. As such, by direct inspection $L[\vp']=0$. It is well-known that $L$ has a simple negative eigenvalue and the second eigenvalue zero is also simple. Finally, by Weinstein's lemma, Lemma \ref{le:Wei}, the property $L|_{\{\vp, \vp'\}^\perp}>0$ is also well-known, and it reduces to checking  that $\dpr{L^{-1}\vp}{\vp}<0$. This is incidentally equivalent to the stability of the $dn$ solution of the cubic focussing NLS, also classical fact. 
	
\end{proof}

\begin{proposition}
	\label{prop:20}
	For all values of the parameters as stated in Proposition \ref{prop:10}, the scalar Schr\"odinger operator
	$$
	\cq:=\cl_1-\f{1}{2A}=-\al\p_x^2+\left(c-1-\f{1}{2A}\right) - 2\be \psi, \ \    D(\cq)=H^2_{per.}[-T,T],
	$$
	 has the following spectral properties
	\begin{enumerate}
		\item $\cq$ has exactly one negative eigenvalue, with an eigenfunction,  say $\psi_0>0$.
		\item $\cq[\psi']=0$. Moreover $Ker(\cq)=span[\psi']$ and $\cq|_{\{\psi_0, \psi'\}^\perp}>0$.
		\item $\cq|_{\{\psi\}^\perp}\geq 0$ and moreover\footnote{In an earlier attempt at resolving this problem, it was tempting to verify that  $\cq|_{\{1\}^\perp}>0$. As it turns out, this would have been enough for our purposes.  
			Unfortunately,  this turned out to be false. In fact,  we have checked that  $\dpr{\cq^{-1} 1}{1}>0$, thus violating Weinstein's condition in Lemma \ref{le:Wei}. }, $\cq|_{\{\psi, \psi'\}^\perp}>0$. That is, there exists $\de>0$, so that
		$$
		\dpr{\cq h}{h}\geq \de\|h\|^2, \forall h: h\perp \psi, h\perp \psi'
		$$
	\end{enumerate}
In particular, if $h\perp \psi$ (equivalently $h\perp \vp^2$),  and $\dpr{\cq h}{h}=0$, then $h=const. \psi'$.
\end{proposition}

\begin{proof}
	Using the relations in Proposition \ref{prop:10} - note that 
	$$
	\left(c-1-\f{1}{2A}\right) =4\sl \si, \ \ 2\be=\f{12\al \si}{2-\ka^2},
	$$
	 we can rewrite the operator $\cq$ as follows
	$$
	\cq = \al\left[-\p_x^2+4 \si - \f{12\si }{2-\ka^2} dn^2
	(\sqrt{\f{\si}{2-\ka^2}} x, \ka)\right]
	$$
	Dilating appropriately, matters reduce to showing that the Schr\"odinger operator
	$$
	Q=-\p_y^2+4(2-\ka^2)-12 dn^2(y, \ka)=-\p_y^2+12\kappa^2sn^2(y,\kappa)-4(1+\kappa^2),
	$$
	has the properties outlined in the statement. That is, we need to show that $Q$ has exactly one negative eigenvalue,  $Ker[Q]=span[(dn^2)']$, and finally $Q|_{\{dn^2, (dn^2)'\}^\perp}>0$. The first three(simple) eigenvalues
and corresponding eigenfunctions of $Q$ are
  $$\begin{array}{ll}
    \mu_0=k^2-2-2\sqrt{1-k^2+4k^4}<0, \\
    \psi_0(y)=dn(y;k)[1-(1+2k^2-\sqrt{1-k^2+4k^4})sn^2(y;k)]>0\\
    \\
    \mu_1=0\\
    \psi_1(y)=dn(y;k)sn(y;k)cn(y;k)={\frac{1}{2}}{\frac{d}{dy}}cn^2(y;k)\\
    \\
    \mu_2=k^2-2+2\sqrt{1-k^2+4k^4}>0\\
    \psi_2(y)=dn(y;k)[1-(1+2k^2+\sqrt{1-k^2+4k^4})sn^2(y;k)].
   \end{array}
  $$
  Since the eigenvalues of $\cq$ and $Q$ are related by
  $\lambda_n=\alpha\gamma^2 \mu_n$, it follows that the first three eigenvalues of the operator
      $\cq$, equipped with periodic boundary condition on $[-T,T]$
      are simple and $\lambda_0<0, \lambda_1=0, \lambda_2>0$. The
      corresponding eigenfunctions are $\psi_0(\gamma x),
      \psi_1(\gamma x)=const. (dn^2)'$ and $\psi_2(\gamma x)$.
	
	By Weinstein's lemma, and since we already know that $n(\cq)=1$, it will suffice to show
	\begin{equation}
		\label{65}
		\dpr{Q^{-1} [dn^2]}{dn^2}<0.
	\end{equation}
	We start with the construction of the Green's function of $\cq$. To this end,  $Q[dn^2(y,\kappa)']=0$. The function
	$$\phi(y)=[dn^2(y,\kappa)]'\int^{x}\frac{1}{[dn^2(y,\kappa)]'^2}dy$$
	is also a solution of $Q\phi=0$. Since $[dn^2(y,\kappa)]'$ has zeros, using the identities
	$$\frac{1}{cn^2y}=\frac{1}{dn(y)}\frac{\partial}{\partial y}\frac{sn(y)}{cn(y)}, \; \; \frac{1}{sn^2y}=-\frac{1}{dn(y)}\frac{\partial}{\partial y}\frac{cn(y)}{sn(y)}$$
	and integrating by parts yields
	$$\begin{array}{ll}\phi(y)&=-\frac{sn(y)cn(y)dn(y)}{2\kappa^2}[\frac{sn(y)}{cn(y)dn^3(y)}-\frac{cn(y)}{sn(y)dn^3(y)}-
		\int\frac{3\kappa^2sn^2(y)cn(y)}{cn(y)dn^4(y)}dy+\int\frac{3\kappa^2sn(y)cn^2(y)}{sn(y)dn^4(y)}]\\
		&=\frac{1}{2\kappa^2}\left[ \frac{1-2sn^2(y)}{dn^2(y)}-3\kappa^2sn(y)cn(y)dn(y)\int\frac{1-2sn^2(y)}{dn^4(y)}dy\right]
	\end{array}$$
	Note that
	$$\det\begin{pmatrix} (dn^2y)'& \phi\\(dn^2y)''&\phi'\end{pmatrix}=1.$$
	Thus, the Green’s function is as follows
	$$Q^{-1}f=(dn^2(y))'\int_{0}^{y}\phi(s)f(s)ds-\phi(y)\int_{0}^{y}(dn^2(s))'f(s)ds+C_f\phi(y),$$
	where $C_f$ is chosen, so that $Q^{-1}f$ has the period $2K(\kappa)$. Now
	$$\begin{array}{ll} \langle Q^{-1}[dn^2],dn^2 \rangle &=\dpr{(dn^2y)'\int_{0}^{y}\phi(s)dn^2(s)ds}{dn^2y}-\dpr{\phi(y)\int_{0}^{y}(dn^2s)'dn(s) ds}{dn^2y}\\
		&+C\dpr{\phi(y)}{dn^2y}. \end{array}$$
	Integrating by parts, we get
	$$\langle Q^{-1}[dn^2],dn^2\rangle=-\dpr{\phi(y)}{dn^6y}+\left( \frac{dn^4(K(\kappa))}{2}+\frac{dn^4(0)}{2}+C\right)\dpr{\phi(y)}{dn^2y}.$$
	By direct calculations, we have that
	$$
	C_{dn^2}=-\frac{(1-\kappa^2)\kappa^2}{\phi'(K(\kappa))}\langle \phi(y),dn^2y\rangle-\frac{(2-\kappa^2)\kappa^2}{2}. 
	$$
	Plugging in the above equation, we get
	\begin{equation}\label{es:1}
		\langle Q^{-1}[dn^2y],dn^2y\rangle=-\dpr{\phi(y)}{dn^6y}+\left[(1-\kappa^2)^2-\frac{\kappa^2(1-\kappa^2)}{\phi'(K(\kappa))}\dpr{\phi(y)}{dn^2y}\right]\dpr{\phi(y)}{dn^2y}.
	\end{equation}
	By direct estimates, we have
	$$\phi'(K(\kappa))=\frac{3(1-\kappa^2)}{2}\int_0^{K(\kappa)}\frac{1-2sn^2y}{dn^4y}dy.$$
	Integrating by parts, we get
	$$\langle \phi(y), dn^2y\rangle=\frac{1}{8\kappa^2}\left[ \int_{-K(\kappa)}^{K(\kappa)}(1-2sn^2y)dy+3(1-\kappa^2)^2\int_{-K(\kappa)}^{K(\kappa)}\frac{1-2sn^2y}{dn^4y}dy\right]$$
	and
	$$\langle \phi(y), dn^6y\rangle=\frac{1}{16\kappa^2}\left[ 5\int_{-K(\kappa)}^{K(\kappa)}(1-2sn^2y)dn^4ydy+3(1-\kappa^2)^4\int_{-K(\kappa)}^{K(\kappa)}\frac{1-2sn^2y}{dn^4y}dy\right]$$
	Using that 
	$$\left\{ \begin{array}{ll}
		\kappa^2sn^2y+dn^2y=1\\
		\\
		\int_{0}^{K(\kappa)}\frac{1}{dn^2y}dy=\frac{1}{1-\kappa^2}E(\kappa)
		\\
		\\
		\int_{0}^{K(\kappa)}\frac{1}{dn^4y}dy=\frac{1}{3(1-\kappa^2)^2}[2(2-\kappa^2)E(\kappa)-(1-\kappa^2)K(\kappa)]\\
		\\
		\int_{0}^{K(\kappa)}dn^4ydy=\frac{4-2\kappa^2}{3}E(\kappa)-\frac{1-\kappa^2}{3}K(\kappa)\\
		\\
		\int_{0}^{K(\kappa)}dn^6ydy=\frac{8\kappa^4-23\kappa^2+23}{15}E(\kappa)-\frac{4(1-\kappa^2)(2-\kappa^2)}{15}K(\kappa),
	\end{array} \right.
	$$
	we get
	$$\left\{ \begin{array}{ll}
		\int_{-K(\kappa)}^{K(\kappa)}\frac{1-2sn^2y}{dn^2y}dy=\frac{2}{\kappa^2(1-\kappa^2)}[2(1-\kappa^2)K(\kappa)-(2-\kappa^2)E(\kappa)]\\
		\\
		\int_{-K(\kappa)}^{K(\kappa)}\frac{1-2sn^2y}{dn^4y}dy=\frac{2}{3\kappa^2(1-\kappa^2)^2}[(2-\kappa^2)(1-\kappa^2)K(\kappa)-2(1-\kappa^2+\kappa^4)E(\kappa)].
		\\
		\\
		\langle \phi(y), dn^2y\rangle=\frac{1}{4\kappa^2}[2(1-\kappa^2)E(\kappa)-(2-\kappa^2)K(\kappa)]\\
		\\
		\langle \phi(y),dn^6y\rangle=\frac{2-\kappa^2}{8\kappa^2}\left[ 2(1-\kappa^2+\kappa^4)E(\kappa)-(1-\kappa^2)(2-\kappa^2)K(\kappa)\right]\\
		\\
		\phi'(K(\kappa))=\frac{1}{2\kappa^2(1-\kappa^2)}[(2-\kappa^2)(1-\kappa^2)K(\kappa)-2(1-\kappa^2+\kappa^4)E(\kappa)].
	\end{array} \right.
	$$
	Putting all together in (\ref{es:1}), we get
	\begin{equation}\label{es:2}
		\begin{array}{ll} \langle Q^{-1}[dn^2],dn^2\rangle &=\frac{2-\kappa^2}{8\kappa^2}[(1-\kappa^2)(2-\kappa^2)K(\kappa)-2(1-\kappa^2+\kappa^4)E(\kappa)]\\
			\\
			&+\frac{(1-\kappa^2)^2}{8\kappa^2}\frac{[(2-\kappa^2)^2K(\kappa)-2(2-\kappa^2+\kappa^4)E(\kappa)][2(1-\kappa^2)E(\kappa)-(2-\kappa^2)K(\kappa)]}{(1-\kappa^2)(2-\kappa^2)K(\kappa)-2(1-\kappa^2+\kappa^4)E(\kappa)}.
		\end{array}
	\end{equation}
As is obvious from Figure \ref{fig:Pic1}, this expression is negative for all values of the parameters $\ka$, whence the claim \eqref{65} is verified. 
 \begin{figure}
	\centering
	\includegraphics[width=0.7\linewidth]{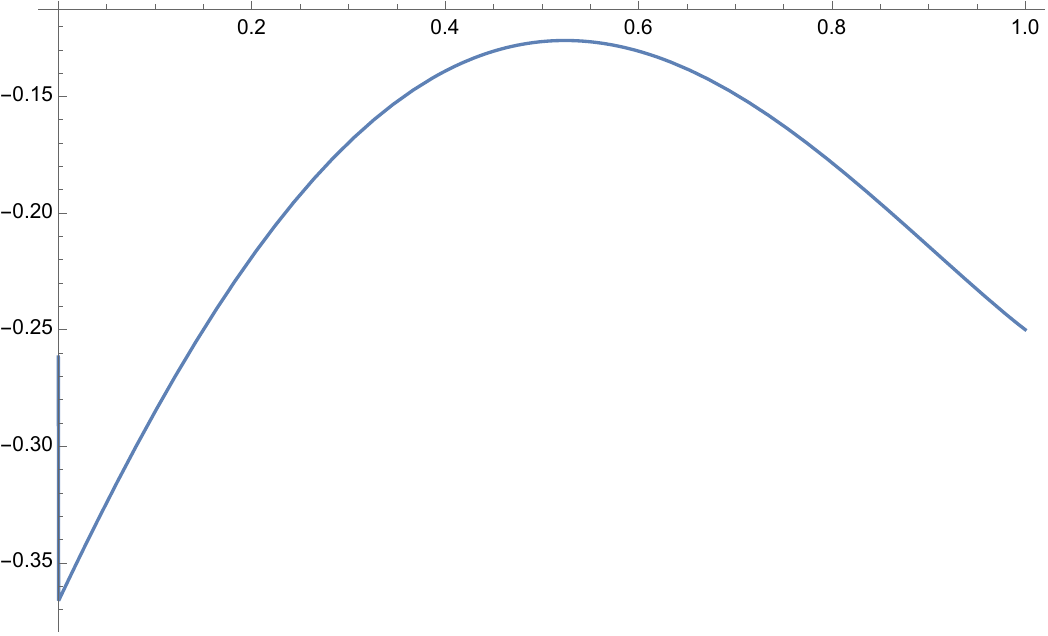}
	\caption{Graph of 	$\ka\to  \langle Q^{-1}[dn^2],dn^2\rangle$}
	\label{fig:Pic1}
\end{figure}
\end{proof}

\section{Analysis of the eigenvalue problem \eqref{e:10}} 
\label{sec:3} 
We now work with the eigenvalue problem \eqref{e:10} directly. The aim is to extract some extra orthogonality relations, which will help in our goal to establish spectral stability. Specifically, we write \eqref{e:10} 
\begin{equation}
	\label{400} 
	\left\{  \begin{array}{l} 
	\p_x(\cl_1 f_1 - \vp f_2)=\la f_1 \\
	-(-\vp f_1+\cl_2 f_2)=\la f_3\\
		\cl_2 f_3=\la f_2. 
	\end{array}
	\right.
\end{equation}
Specifically, we are looking to show that there is no non-trivial solution to \eqref{400}, for any $\la>0$.  We assume, for a contradiction,  there is such a solution $\vec{f}$ for some fixed $\la>0$. 

From these relations, by taking dot product with $1$ in the first equation and $\vp$ in the third one, we immediately see that $f_1\perp 1$ and $f_2\perp \vp$ (recall $\cl_2 \vp=0$). Clearly, since $\p_x$ annihilates constants, $\p_x=\p_x P_{\{1\}^\perp}$. Also, upon introducing $P_{\{\vp\}^\perp}$ in the second and third equations (note that $P_{\{\vp\}^\perp} \cl_2=\cl_2$), we arrive at 
\begin{equation}
	\label{410} 
	\left\{  \begin{array}{l} 
		\p_x P_{\{1\}^\perp} (\cl_1 f_1 - \vp f_2)=\la f_1 \\
		-P_{\{\vp\}^\perp} (-\vp f_1+\cl_2 f_2)=\la \tilde{f}_3\\
		P_{\{\vp\}^\perp}  \cl_2 \tilde{f}_3=\la f_2. 
	\end{array}
	\right.
\end{equation}
where $\tilde{f}_3=P_{\{\vp\}^\perp} f_3\perp \vp$ (note that $\cl_2 f_3=\cl_2 \tilde{f}_3$, since $f_3-\tilde{f}_3\in span[\vp]=Ker(\cl_2)$). The form \eqref{410}, while equivalent\footnote{Clearly having non-trivial solutions of \eqref{400} implies the existence of non-trivial solutions to \eqref{410}. Indeed, the only other possibility is that $\tilde{f}_3=0$, but this clearly implies $f_2=0$ and then from the second equation $\vp f_1=const. \vp$, whence, since $\dpr{f_1}{1}=0$, it follows $f_1=0$. From there,  it follows that  $f_3=0$, a contradiction} to \eqref{400}, yields an  important restriction of the eigenvalue problem, which are useful in the sequel. Indeed, we may now rewrite \eqref{410} 
$$
\cj \left(\begin{matrix}
	P_{\{1\}^\perp} & 0 & 0 \\
	0 & P_{\{\vp\}^\perp} & 0 \\
	0 & 0 & P_{\{\vp\}^\perp}
\end{matrix}\right)\cl \left(\begin{matrix}
P_{\{1\}^\perp} & 0 & 0 \\
0 & P_{\{\vp\}^\perp} & 0 \\
0 & 0 & P_{\{\vp\}^\perp}
\end{matrix}\right) \left(\begin{matrix}
f_1 \\ f_2 \\ \tilde{f}_3
\end{matrix}\right)=\la  \left(\begin{matrix}
f_1 \\ f_2 \\ \tilde{f}_3
\end{matrix}\right). 
$$
So, it makes sense to introduce the restricted self-adjoint operator 
$$
\ch:=\left(\begin{matrix}
	P_{\{1\}^\perp} & 0 & 0 \\
	0 & P_{\{\vp\}^\perp} & 0 \\
	0 & 0 & P_{\{\vp\}^\perp}
\end{matrix}\right)\cl \left(\begin{matrix}
	P_{\{1\}^\perp} & 0 & 0 \\
	0 & P_{\{\vp\}^\perp} & 0 \\
	0 & 0 & P_{\{\vp\}^\perp}
\end{matrix}\right). 
$$
The goal now is to disprove that the eigenvalue problem 
\begin{equation}
	\label{420} 
	\cj \ch \vec{f}=\la \vec{f}
\end{equation}
with $\la>0$, does not have any non-trivial solutions. 
\begin{proposition}
	\label{prop:34} 
	The operator $\ch$ is non-negative. 
\end{proposition}
\begin{proof}
Standard results for differential operators with periodic coefficients show that spectrum of $\ch$ consists entirely of eigenvalues. We proceed with a proof by contradiction - namely, we aim at showing that $\ch$ does not possess a negative eigenvalue. 

 To this end, assume that there is a negative eigenvalue, say $-\nu^2$. We clearly see that $f_3=0$, as $\cl_2\geq 0$, so we concentrate on the first two entries.  We have  
\begin{equation}
	\label{110}
	\begin{cases}
		\cl_1 f_1 - \vp f_2=-\nu^2 f_1 \\
		-\vp f_1+\cl_2 f_2=-\nu^2 f_2,
	\end{cases}
\end{equation}
where without loss of generality, we may assume that $f_1, f_2$ are real-valued. 
Recalling $f_1\perp 1, f_2\perp \vp$, we take dot product with $\vp$ in the second equation, we have 
$$
\dpr{\vp f_1}{\vp}=\dpr{\cl_2 f_2}{\vp}+\nu^2 \dpr{f_2}{\vp}=
\dpr{f_2}{\cl_2 \vp}=0.
$$
Hence, $f_1\perp \vp^2$ as well. We proceed by taking dot product of \eqref{110} with 
$\left(\begin{matrix}
f_1 \\ 
	f_2
\end{matrix}\right)$. We obtain 
$$
-\nu^2 (\|f_1\|^2+\|f_2\|^2)=\dpr{\cl_1 f_1}{f_1} -2 \dpr{\vp f_1}{f_2}+\dpr{\cl_2 f_2}{f_2} 
$$
We estimate by Cauchy-Schwartz $- 2\dpr{\vp f_1}{f_2}\geq -2A \dpr{\vp^2 g_2}{g_2}-\f{1}{2A} \|f_1\|^2$, so that 
\begin{equation}
	\label{140}
	0> -\nu^2 (\|f_1\|^2+\|f_2\|^2)\geq \dpr{\left(\cl_1-\f{1}{2A}\right) f_1}{f_1} + \dpr{\left(\cl_2 - 2 A^2 \vp^2\right) f_2}{f_2}=\dpr{\cq f_1}{f_1}+\dpr{L f_2}{f_2}
\end{equation}
Recall now that by Proposition \ref{prop:20}, $\cq|_{\{\psi\}^\perp}\geq 0$, so $\dpr{\cq f_1}{f_1}\geq 0$, as $f_1\perp \vp^2=const. \psi$. Furthermore, by Proposition \ref{prop:15}, we have that $L|_{\{\vp\}^\perp}\geq 0$, whence $\dpr{L f_2}{f_2}\geq 0$, as $f_2\perp \vp$. This clearly demonstrates that \eqref{140} is contradictory, which means that $\ch\geq 0$. 
\end{proof}
It is now trivial to finish the spectral stability proof for the dnoidal waves. The following proposition is standard, yet we provide its short proof for completeness. 
\begin{proposition}
	\label{prop:38} 
	The eigenvalue problem \eqref{420} does not have non-trivial solutions with $\Re\la>0$. 
\end{proposition}
{\bf Remark:} In other words, the dnoidal waves are stable, as we have reduced matters to \eqref{420}. 
\begin{proof}
	Indeed, consider a solution of \eqref{420}. Take a dot product with $\ch \vec{f}$. We have 
	\begin{equation}
		\label{440} 
			\dpr{\cj \ch \vec{f}}{\ch \vec{f}}=\la \dpr{\vec{f}}{\ch \vec{f}}. 
	\end{equation}
	Clearly, as $\cj^*=-\cj$, $	\dpr{\cj \ch \vec{f}}{\ch \vec{f}}=-	\dpr{ \ch \vec{f}}{\cj \ch \vec{f}}=-\overline{	\dpr{\cj \ch \vec{f}}{\ch \vec{f}}}$, whence $\Re 	\dpr{\cj \ch \vec{f}}{\ch \vec{f}}=0$. 
	
	On the other hand, $\dpr{\vec{f}}{\ch \vec{f}}=\dpr{\ch \vec{f}}{ \vec{f}}=\overline{\dpr{\vec{f}}{\ch \vec{f}}}$, whence $\dpr{\vec{f}}{\ch \vec{f}}$ is real. Taking real parts in \eqref{440}, we obtain 
	$$
	0=\dpr{\ch \vec{f}}{ \vec{f}} \Re \la.  
	$$
	It remains to only  show that  $\ch\geq 0$, $\dpr{\ch \vec{f}}{ \vec{f}}>0$. Indeed, assuming for a contradiction $\dpr{\ch \vec{f}}{ \vec{f}}=0$, it follows that $\ch \vec{f}=0$, which from \eqref{420} would mean $\vec{f}=0$, a contradiction. Hence, we conclude that $\Re \la=0$, which means that the whole spectrum of $\cj\ch$  is purely imaginary, hence stability.  
\end{proof}

\appendix

\section{Proof of Proposition \ref{prop:10}}

    Integrating once the equation (\ref{3.3}), we get
    \begin{equation}\label{3.4}
    	\varphi'^2=\frac{A}{2}\left[ -\varphi^4-\frac{2\sigma}{A}\varphi^2+a\right],
    \end{equation}
    where $a$ is a constant of integration. In the case $A>0, \si>0$,  let $\varphi_0>\varphi_1>0$ are the roots of polynomial $-r^4+\frac{2\sigma}{A}r^2+a$. Then the equation (\ref{3.4}) can be written in the form
    $$
    \varphi'^2=\frac{A}{2}(\varphi_0^2-\varphi^2)(\varphi^2-\varphi_1^2).
    $$
    The solution $\varphi$ of equation (\ref{3.4}) is given by
    \begin{equation}\label{3.5}
    	\varphi(x)=\varphi_0dn(\gamma x, \kappa),
    \end{equation}
    where
    \begin{equation}\label{3.6}
    	\kappa^2=\frac{\varphi_0^2-\varphi_1^2}{\varphi_0^2}=\frac{2\varphi_0^2-\frac{2\sigma}{A}}{\varphi_0^2}, \; \; \; \gamma^2=\frac{A\varphi_0^2}{2}=\frac{\sigma}{2-\kappa^2}.
    \end{equation}
    From the second equation of system  (\ref{3.2}) and $\psi=A\varphi_0^2dn^2(\gamma x, \kappa)$, we get the relation
    $$\begin{array}{ll} &[(1-c)A\varphi_0^2-4A\alpha (\kappa^2-2)\varphi_0^2\gamma^2
    	+\frac{1}{2}\varphi_0^2]dn^2(\gamma x, \kappa)\\
    	\\
    	&+(\beta A^2\varphi_0^4-6\alpha A\varphi_0^2\gamma^2)dn^4(\gamma x, \kappa)-2\alpha A(1-\kappa^2)\varphi_0^2\gamma^2=b.
    \end{array}
    $$
    Since the functions $1, dn^2(y,\ka), dn^4(y,\ka)$ are linearly independent, it follows, by coefficient matching, that the following are satisfied
    \begin{eqnarray*}
    	\be=3\al, \ \
    	A=\frac{1}{2(c-1-4\alpha \sigma)}.
    \end{eqnarray*}
    It follows that $c-1-4\al \si=\f{1}{2A}>0$.


\end{document}